\documentclass[12pt]{amsart} 
\usepackage{tikz-cd}
\usepackage{fancyhdr} 
\usepackage{latexsym, amssymb, amscd, amsfonts,pstricks,enumitem,amsmath, young}
\usepackage{latexsym, amssymb, amscd, amsfonts,pstricks,enumitem,amsmath}
\usepackage{amsmath,amsthm,amssymb,mathrsfs}
 
\usepackage[all]{xypic}
\usepackage{ifthen}
\usepackage{longtable}
 
\numberwithin{equation}{section}

\providecommand{\binom}[2]{{#1\choose#2}}

\newcommand{\Hom}{\operatorname{Hom}}
\newcommand{\Osh}{{\mathcal O}}                        %  Structure sheaf
\renewcommand{\H}{\mathrm{H}}                          %  Cohomology group
\renewcommand{\P}{\mathrm{P}}
\newcommand{\Pic}{\operatorname{Pic}} %Pic
\newcommand{\E}{\mathrm{E}}                        
\newcommand{\K}{\mathrm{K}}                            
\newcommand{\Br}{\operatorname{Br}}
\newcommand{\cchar}{\operatorname{char}}
\newcommand{\G}{\mathbb{G}}
\newcommand{\NS}{\operatorname{NS}} %Neron-Severi
\newcommand{\kk}{\mathbf{k}}
\newcommand{\GL}{\operatorname{GL}}
\newcommand{\PGL}{\operatorname{PGL}}
\newcommand{\spec}{\operatorname{Spec}}
\newcommand{\CC}{\mathbb{C}} % complex numbers
\newcommand{\PP}{\mathbb{P}} % projective space
\newcommand{\QQ}{\mathbb{Q}} % rational numbers
\newcommand{\ZZ}{\mathbb{Z}} % integers

\newtheorem{theorem}{Theorem}[section]
\newtheorem{lemma}[theorem]{Lemma}
\newtheorem{corollary}[theorem]{Corollary}
\newtheorem{proposition}[theorem]{Proposition}

\theoremstyle{definition}
\newtheorem{defn}[theorem]{Definition}
\newtheorem{remark}[theorem]{Remark}
\newtheorem{example}[theorem]{Example}

\begin{document}
\title[Cubic torsors and projective bundles over Abelian varieties]{On cubic torsors, biextensions and Severi-Brauer varieties over Abelian varieties}
%\title{}
%\author{}
\author{Nathan Grieve}
{\blue 
\address{
School of Mathematics and Statistics, 4302 Herzberg Laboratories, Carleton University, 1125 Colonel By Drive, Ottawa, ON, K1S 5B6, Canada
}
\address{D\'{e}partement de math\'{e}matiques, Universit\'{e} du Qu\'{e}bec \'a Montr\'{e}al, Local PK-5151, 201 Avenue du Pr\'{e}sident-Kennedy, Montr\'{e}al, QC, H2X 3Y7, Canada}

\email{nathan.m.grieve@gmail.com}%

}

\begin{abstract} 
We study homogeneous irreducible Severi-Brauer varieties over an Abelian variety $A$.  Such objects were classified by Brion \cite{Bri}.  Here we interpret that result within the context of cubic structures and biextensions for certain $\G_m$-torsors over finite subgroups of $A$.  Our results build on the theory of Breen, \cite{Breen:1983}, and Moret-Bailly \cite{Moret-Bailly}.
\end{abstract}
\thanks{This article has been published in S\~{a}o Paulo Journal of Mathematical Sciences. The final published version is available online at:  https://doi.org/10.1007/s40863-021-00250-3.}

\maketitle

\section{Introduction}

Let $A$ be an Abelian variety, over an algebraically closed field $\kk$, and let 
$$\Br(A) := \H^2_{\mathrm{et}}(A, \G_m)$$ 
be its cohomological Brauer group.  Fix a positive integer $r$ with  $p:= \cchar(\kk) \nmid r$.  Our purpose here is to study $\Br(A)[r]$, the $r$-torsion subgroup of $\Br(A)$.   In doing so, we apply work of Brion, \cite{Bri}, which builds on previous work of Atiyah, \cite{Atiyah}, Mumford \cite{Mum:v1}, Berkovich \cite{Berkovich:1972}, Mukai, \cite{Muk78}, and others. 

As in \cite{Bri}, we say that a Severi-Brauer variety over $A$ is \emph{homogeneous} if it is invariant under pullback by all translations of elements of $A$.  A homogeneous Severi-Brauer variety is \emph{irreducible} if it contains no nontrivial proper homogeneous subbundle.
In Section \ref{proof:Brion:thm}, we give a complete proof of the following result which was observed by Brion.  It applies earlier work of Berkovich \cite{Berkovich:1972}.

\begin{theorem}[{\cite[page 2497]{Bri}}]\label{Brion:Br:thm}
Each class in $\Br(A)[r]$ may be represented by a homogeneous irreducible Severi-Brauer variety.
\end{theorem}

Second, we study \emph{cubic structures} on $\G_m$-torsors over subgroups $\K \subseteq A[n]$, of the group of $n$-torsion points of $A$.  We require that $\# \K = n^2$ and $p \nmid n$.  Our result identifies a class of such cubic torsors which correspond to the homogeneous irreducible $\PP^{n-1}$-bundles over $A$.  We say that such cubic torsors are \emph{nondegenerate} and refer to Definition \ref{non-degenerate-couples} for precise details.  Using this terminology, our main novel conceptual result is

\begin{theorem}\label{Br:thm:2}
Each class in $\Br(A)[r]$ may be represented by a nondegenerate cubic torsor.
\end{theorem}

One step in our proof of Theorem \ref{Br:thm:2} involves a reformulation of the basic structure theorems, for homogeneous projective bundles over $A$, which were obtained by Brion in \cite{Bri}.  
More precisely, a inite type group scheme is \emph{anti-affine} if it admits no global non-constant regular functions.  In particular, such a group is smooth, connected and commutative.  For our purposes, an \emph{anti-affine extension} of an Abelian variety $A$, is a short exact sequence of commutative group schemes 
\begin{equation}\label{anti:affine:ext:intro:1}
1 \rightarrow H \rightarrow G \xrightarrow{\pi} A \rightarrow 1
\end{equation}
in which the group $G$ is anti-affine and $H$ is affine.

In \cite{Bri}, Brion characterized the homogeneous projective bundles over $A$, in terms of anti-affine extensions \eqref{anti:affine:ext:intro:1} together with a faithful projective representation of the commutative group scheme $H$.

\begin{theorem}[{\cite[Theorem 2.1]{Bri}}]\label{homog:characterize}  Let $\kk$ be an algebraically closed field.
Let $\P$ be a $\PP^{n-1}$-bundle over an Abelian variety $A$ and assume that $n$ is not divisible by the characteristic of $\kk$.  The following assertions hold true.
\begin{itemize}
\item[(i)]{
The $\PP^{n-1}$-bundle $\P$ is homogeneous if and only if there exists an anti-affine extension \eqref{anti:affine:ext:intro:1}
together with a faithful homomorphism
\begin{equation}\label{anti:affine:ext:rep:1}
\rho \colon H \rightarrow \operatorname{PGL}_n
\end{equation}
such that $\P$ is the associated bundle
\begin{equation}\label{anti:affine:bundle}
\P \simeq G \times^H \PP^{n-1} \rightarrow G / H = A.
\end{equation}
Here, in \eqref{anti:affine:bundle}, $H$ acts on $\PP^{n-1}$ via the representation \eqref{anti:affine:ext:rep:1}.  The anti-affine extension \eqref{anti:affine:ext:intro:1} is unique and the projective representation \eqref{anti:affine:ext:rep:1} is unique up to conjugacy in $\operatorname{PGL}_n$.  Further, the $\PGL_n$-torsor that corresponds to the bundle \eqref{anti:affine:bundle} is the $\PGL_n$-bundle
$$
G \times^H \PGL_n \rightarrow A \text{;}
$$
the corresponding Azumaya algebra is
$$
\mathcal{A} = \left(\pi_*(\Osh_G) \otimes M_n(\Osh_A) \right)^H \text{.}
$$
Finally, the $A$-automorphism group of such a homogeneous bundle \eqref{anti:affine:bundle}  is isomorphic to the centralizer of $H$ in the projective linear group $\operatorname{PGL}_n$. 
}
\item[(ii)]{If $\P$ is homogeneous, then its homogeneous projective subbundles are exactly the bundles 
$$
G \times^HS \rightarrow A
$$
where $S \subseteq \PP^{n-1}$ is an $H$-stable linear subspace.  Further, if $(\P_1,\P_2)$ is a pair of disjoint $\PP^{n_i-1}$-subbundles of $\P$, with $n_1+n_2 = n$, then both $\P_1$ and $\P_2$ are homogeneous subbundles of $\P$.
}
\end{itemize}
\end{theorem}

In the present article, we relate the conceptual framework of  \cite{Breen:1983} and \cite{Moret-Bailly}  to the classification that is given in \cite{Bri}.  In doing so, we obtain a description of $\operatorname{Br}(A)[r]$ that is in terms of cubic structures and biextensions.  
The fact that the Theorem of the Cube is related to the Brauer group is well-known \cite{Hoobler:1972}.  We refer to the works  \cite{Skorobogatov:Zarhin:2008} and \cite{Brion:2020}, for example, as a starting point to the study of Brauer groups of Abelian varieties over non-algebraically closed fields.  We also mention the work \cite{Poonen:Rains:2011} which provides additional motivation for further research in the direction that we pursue here.

Our main result is Theorem \ref{irred:homog:theta:bundles:thm}.  It interprets the homogeneous irreducible $\PP^{n-1}$-bundles over $A$, for $p \nmid n$, in terms of the works \cite{Breen:1983} and \cite{Moret-Bailly}.  We also use Theorem \ref{irred:homog:theta:bundles:thm}, in conjunction with Theorem \ref{Brion:Br:thm}, to prove Theorem \ref{Br:thm:2}.  Both of these results, Theorem \ref{irred:homog:theta:bundles:thm} and Theorem \ref{Br:thm:2}, together with Theorem \ref{Brion:Br:thm}, which is due to Brion, have not, at present, been adequately discussed in the literature.

\subsection*{Notations and other conventions}   We fix an algebraically closed base field $\kk$ and let $p := \cchar(\kk)$.  Unless explicitly stated otherwise, we fix positive integers $n$ and $r$ and assume that they are not divisible by $p$.  When no confusion is likely, we fail to distinguish between the concepts of $\G_m$-torsors, total spaces of line bundles and invertible sheaves.  We consider Abelian varieties $A$ over $\kk$.  The dual of $A$ is denoted by $A^t$.  If $x \in A$, then $\tau_x \colon A \rightarrow A$ denotes translation by $x$ in the group law.  If $m \in \ZZ$, then $A[m]$ denotes the kernel of the morphism $[m]_A \colon A \rightarrow A$, $x \mapsto m x$.  This is the set of $m$-torsion points of $A$.   By  a $\PP^{n-1}$-bundle over $A$, or equivalently a rank $n$ Severi-Brauer variety over $A$, we mean a proper flat morphism $\P \rightarrow A$ with fibers at all closed points, $x \in A$, projective space $\PP^{n-1}$.  Finally, if $H$ is a finite commutative $\kk$-group scheme, then $H^t$ is its Cartier dual.

\subsection*{Acknowledgements}  
Portions of this work were completed during my time as a postdoctoral fellow at Michigan State University.  It benefited from visits to ICERM (Providence), during June 2019, and BIRS (Banff), during November 2019.  I thank Igor Rapinchuk and colleagues for their interest and discussions on topics related to this work.   Further, I thank the Natural Sciences and Engineering Research Council of Canada for supporting my work through my grants DGECR-2021-00218 and RGPIN-2021-03821.  Finally, I thank anonymous referees for their careful reading, corrections, helpful comments and suggestions.  

\section{Preliminaries about anti-affine extensions}

In general, the structure of each such anti-affine extension \eqref{anti:affine:ext:intro:1} depends on the characteristic of the base field.  Indeed, the following theorem, from \cite{Brion:2009}, gives a description of the set of isomorphism classes of anti-affine groups over a given Abelian variety.  In what follows, by an \emph{anti-affine group} over an Abelian variety $A$, is meant an exact sequence of commutative group schemes 
$$
1 \rightarrow H \rightarrow G \rightarrow A \rightarrow 1 \text{,} 
$$
 as in \eqref{anti:affine:ext:intro:1}, with $G$ an anti-affine group and $H$ an affine group.

\begin{theorem}[{\cite[Theorem 2.7]{Brion:2009}}]\label{Brion:anti:affine}
Let $A$ be an Abelian variety over an algebraically closed field $\kk$ with $p := \operatorname{char}(\kk)$.  Then, the following assertions hold true.
\begin{enumerate}
\item[(i)]{If $p > 0$, then the set of isomorphism classes of anti-affine groups over $A$ is in bijective correspondence with the set of sublattices of $A^t(\kk)$. }
\item[(ii)]{
If $p = 0$, then the set of isomorphism classes of anti-affine groups over $A$ are in bijective correspondence with the set of pairs $(\Lambda,V)$, where $\Lambda$ is a sublattice of $A^t(\kk)$ and where $V$ is a subvector space of $\H^1(A,\Osh_A)$. 
}
\end{enumerate}
\end{theorem}

While the overall structure of such anti-affine extensions \eqref{anti:affine:ext:intro:1}  depends on $p$, what remains true independently of the characteristic, is that they are classified by specifying a finite subgroup scheme $\K \subseteq A$ together with an anti-affine group over the Abelian variety $(A^t / \K^t)^t$.  
For later use, we make explicit mention of this fact.  We include a proof for completeness.

\begin{proposition}\label{anti:affine:structure:prop}
The anti-affine extensions \eqref{anti:affine:ext:intro:1} are classified by specifying a finite subgroup scheme $\K \subseteq A$ together with an anti-affine group over the Abelian variety $\left(A^t / \K^t\right)^t$.  
Further, given such an anti-affine extension \eqref{anti:affine:ext:intro:1} if $H$ is finite, then the anti-affine group $G$ is an Abelian variety. 
\end{proposition}
\begin{proof}
First, suppose given a finite subgroup scheme $\K \subseteq A$ together with an anti-affine group $G$ over the Abelian variety $(A^t / \K^t)^t$.  Let $H$ be the kernel of the induced map
$$
G \rightarrow (A^t / \K^t)^t \rightarrow A^{tt} = A.
$$
Then we obtain an anti-affine extension of the form \eqref{anti:affine:ext:intro:1}.

For the converse, if $G$ is an anti-affine group over $A$ and fitting into an extension as in \eqref{anti:affine:ext:intro:1},  then the morphism $\pi$ factors through 
$$B := \operatorname{Alb}(G) = G / G_{\mathrm{aff}} \text{, }$$ 
the Albanese variety of $G$.  Let $\K$ be the Cartier dual of $\K^t$, the kernel of the induced isogney 
$$A^t \rightarrow B^t \text{.}$$  
Then $\K$ is the finite subgroup of $A$ that is determined by the anti-affine extension \eqref{anti:affine:ext:intro:1}.  By double duality, arguing as in \eqref{dual:eqn1}, \eqref{dual:eqn2}, \eqref{dual:eqn3}, $G$ is an anti-affine group over $(A^t / \K^t)^t$.

For the remaining assertion, if $H$ is finite, then $G_{\mathrm{aff}}$, the largest  smooth  connected affine normal subgroup of $G$, is trivial.  It follows that $G$ is isomorphic to its Albanese variety.  Hence, $G$ is an Abelian variety. 
\end{proof}

Together, Theorem \ref{homog:characterize} and Proposition \ref{anti:affine:structure:prop} allow for the following characterization of homogenous bundles over $A$.  

\begin{corollary}\label{anti:affine:cor}
Let $\P$ be a rank $n$ Severi-Brauer variety over an Abelian variety $A$ and assume that $n$ is not divisible by the characteristic of $\kk$.  Then $\P$ is homogeneous if and only if there exists a finite subgroup $\K \subseteq A$ and an anti-affine group $G$ over the Abelian variety $(A^t / \K^t)^t$ such that if $H$ is the kernel of the induced map
$
G \rightarrow A \text{,}
$
then 
$$
\P \simeq G \times^H \PP^{n-1} \rightarrow A
$$
for some faithful projective representation 
$
\rho \colon H \rightarrow \operatorname{PGL}_n.
$
\end{corollary}
\begin{proof}
By Proposition \ref{anti:affine:structure:prop},  the anti-affine extensions \eqref{anti:affine:ext:intro:1} are classified by finite subgroups $\K \subseteq A$ together with anti-affine groups $G$ over $(A^t/\K^t)^t$.  Thus, the conclusion that is desired by Corollary \ref{anti:affine:cor} follows upon combining Proposition \ref{anti:affine:structure:prop} and Theorem \ref{homog:characterize} (i).
\end{proof}

A homogeneous projective bundle is irreducible if it contains no nontrival proper homogeneous subbundle.
The classification of homogeneous irreducible projective bundles involves the weight $1$ representations of nondegenerate theta groups, \cite{MumI}, \cite{Moret-Bailly}.  For example, Theorem \ref{homog:characterize}, Proposition \ref{anti:affine:structure:prop} and Corollary \ref{anti:affine:cor}, together with the theory of nondegenerate theta groups, allow for the homogeneous irreducible projective bundles to be described as in Corollary \ref{homog:irred:bundles} below.  Again, we include a short proof for completeness.  We discuss the theory of non-degenerate theta groups in Sections \ref{weight:one:theta} and \ref{S-B-cubical-structures}.

\begin{corollary}[{\cite[Proposition 3.1]{Bri}}]\label{homog:irred:bundles}
The homogeneous irreducible  rank $n$ Severi-Brauer varieties over $A$, for $p \nmid n$, are classified by pairs $(H,e)$, where 
\begin{equation}\label{finite:subgroup}
H \subseteq A[n]
\end{equation}
is a subgroup of rank $n^2$ and where 
\begin{equation}\label{non-deg-alt-pairing}
e \colon H \times H \rightarrow \mathbb{G}_m
\end{equation}
is a nondegenerate alternating pairing.
\end{corollary}
\begin{proof}
By the theory of nondegenerate theta groups, \cite[page  293]{MumI}, see also Sections \ref{weight:one:theta} and \ref{S-B-cubical-structures}, the nondegenerate alternating map \eqref{non-deg-alt-pairing} specifies an irreducible weight one representation of a nondegenerate theta group over $H$.  It also induces an isomorphism 
$$H \simeq H^t\text{.}$$  
The desired anti-affine group over $A$ is the Abelian variety $(A/H^t)^t$.
\end{proof}

\section{Equivalent descriptions of Azumaya algebras and Severi-Brauer Varieties}\label{Azumaya:alg:Ab:Sch}

Let $X$ be a nonsingular projective variety over an algebraically closed field $\kk$ and let $p:=\operatorname{char}(\kk)$.  There are many equivalent descriptions of Azumaya algebras over $X$.  To begin with, by an \emph{Azumaya algebra} over $X$, we mean an $\Osh_X$-algebra $\mathcal{A}$ which is a coherent $\Osh_X$-module and which has the property that its stalk $\mathcal{A}_x$ at all closed points $x \in X$ is an Azumaya algebra over the local ring $\Osh_{X,x}$.  In particular, $\mathcal{A}_x$ is Azumaya over $\Osh_{X,x}$ for all points $x \in X$.   Especially, $\mathcal{A}$ is a locally free $\Osh_X$-module of finite square rank.  Equivalently
\begin{enumerate}
\item[(i)]{the algebra $\mathcal{A}$ is a locally free $\Osh_X$-module and 
$$\mathcal{A}|_x := \mathcal{A}_x \otimes \kappa(x)
$$ 
is a central simple $\kappa(x)$-algebra for each $x \in X$; 
}
\item[(ii)]{the algebra 
$\mathcal{A}$ is a locally free $\Osh_X$-module and the canonical homomorphism
$$ \mathcal{A} \otimes_{\Osh_X} \mathcal{A}^\circ \rightarrow \mathcal{E}nd_{\Osh_X}(\mathcal{A})$$
is an isomorphism; and
}
\item[(iii)]{
there exists a covering $\{U_i \rightarrow X \}_{i \in I}$, for the \'{e}tale topology on $X$, so that for each $i$, there exists $r_i$ with the property that
$$
\mathcal{A} \otimes_{\Osh_X} \Osh_{U_i} \simeq M_{r_i}(\Osh_{U_i}),
$$
}
\end{enumerate}
see for example \cite[Chapter IV]{Milne:Etale}.

The geometric study of Azumaya algebras is achieved via their bijective correspondence with $\PGL_n$-torsors and $\PP^{n-1}$-bundles over $X$.  For instance, under this bijective correspondence, each such rank $n^2$ Azumaya algebra $\mathcal{A}$ has the form
\begin{equation}\label{Azumaya:eqn}
\mathcal{A} \simeq (\pi_*(\Osh_Y) \otimes M_{n}(\Osh_X))^{\PGL_n}
\end{equation}
for
\begin{equation}
\pi \colon Y \rightarrow X
\end{equation}
the corresponding $\PGL_n$-torsor; the corresponding $\PP^{n-1}$-bundle is then
\begin{equation}\label{PP:bundle:structure:morphism}
\P \simeq Y \times^{\PGL_n} \PP^{n-1} \rightarrow X.
\end{equation}
In particular, the structure morphism \eqref{PP:bundle:structure:morphism} is a proper flat morphism with $\PP^{n-1}$-fibres over closed points of $X$.

\begin{example}[{\cite[page  2477]{Bri}}]\label{Severi:Brauer:Segre:Product}
Recall the \emph{product} of two $\PP^{n_i-1}$-bundles 
$$P_i \rightarrow X\text{,}$$ for $i = 1,2$.    Let 
$$Y_i \rightarrow X$$ be the $\PGL_{n_i}$-bundles corresponding to $\P_i$.  Then $$\P_1 \cdot \P_2 \rightarrow X$$
is the $\PP^{n_1 n_2 - 1}$-bundle that is obtained from the $\PGL_{n_1} \times \PGL_{n_2}$-torsor 
$$Y_1 \times_X Y_2 \rightarrow X$$ 
via the tensor product extension of structure groups
$$
\PGL_{n_1} \times \PGL_{n_2} \rightarrow \PGL_{n_1 n_2}.
$$
\end{example}

Next, we consider Example \ref{Severi:Brauer:Segre:Product} for the particular case of homogeneous projective bundles over a given Abelian variety.

\begin{example}[{\cite[Remark 2.2]{Bri}}]  If 
$$\P_i = G_i \times^{H_i} \PP^{n_i - 1}\text{,}$$ 
for $i = 1,2$, are homogeneous $\PP^{n_i-1}$-bundles over and Abelian variety $A$, $p \nmid n_i$, corresponding to anti-affine extensions
$$
1 \rightarrow H_i \rightarrow G_i \rightarrow A \rightarrow 1
$$
and faithful projective representations 
$$
\rho_i \colon H_i \rightarrow \PGL_{n_i} \text{, }
$$
then the product bundle $\P_1 \cdot \P_2$ is homogeneous.  It has the form
$$
\P_1 \cdot \P_2 = G\times^H \PP^{n_1 \cdot n_2 - 1}
$$
for 
$$G \subseteq G_1 \times_A G_2$$ 
the largest anti-affine subgroup, 
$$H := (H_1 \times H_2) \bigcap G$$ 
and projective representation 
$$\rho = (\rho_1 \otimes \rho_2)|_H \text{.}$$
\end{example}

To conclude this section, we recall the cohomological description of Severi-Brauer varieties.  Let $n$ be a positive integer that is not divisible by $p$.
Recall that the pointed cohomology set $\H^1_{\mathrm{et}}(X,\PGL_n)$ classifies isomorphism classes of rank $n^2$ Azumaya algebras over $X$.  Equivalently, $\H^1_{\mathrm{et}}(X,\PGL_n)$ classifies $\PGL_n$-torsors over $X$.  The class of such an Azumaya algebra in 
$$\Br(X) = \H^2_{\mathrm{et}}(X,\G_m)$$ 
is obtained via the boundary morphism
$$
\partial \colon \H^1_{\mathrm{et}}(X,\PGL_n) \rightarrow \Br(X)
$$
which is induced by the exact sequence
$$
1 \rightarrow \G_m \rightarrow \GL_n \rightarrow \PGL_n \rightarrow 1
$$
of \'{e}tale sheaves of groups.

\section{Cyclic Azumaya algebras and the Brauer group}

In this section, we recall the most basic properties of $\Br(X)[r]$, the \emph{$r$-torsion} part of the Brauer group of a nonsingular projective variety $X$.  Here $r$ is a positive integer which is not divisible by the characteristic of the algebraically closed base field $\kk$.  The main point is Theorem \ref{Ab:var:Brauer:presentation}, which, among other things, says that the group $\Br(X)[r]$ is generated by the \emph{Hilbert symbols} $\{\alpha,\beta \}_r$.  Equivalently, $\Br(X)[r]$ is generated by the classes of $\mathcal{A}(\alpha,\beta)$, the cyclic algebras determined by $r$-torsion line bundles $\alpha$ and $\beta$ on $X$.  What we describe here is based on \cite{Berkovich:1972}, \cite{Milne:Etale} and \cite{Kaji:1988}.  Corollary \ref{Ab:var:Brauer:presentation:rank} and Example \ref{E1:times:E2:examples} illustrate the more specific case of Abelian varieties.

By a \emph{rigidified $r$-torsion line bundle} on $X$, is meant an $r$-torsion line bundle $\alpha$ together with a fixed choice of isomorphism 
$$\Osh_X \xrightarrow{\sim} \alpha^{\otimes r} \text{.}$$ 
Let $\mu_r$ denote the multiplicative group of $r$th roots of unity and fix a primitive $r$th root of unity $\zeta \in \mu_r$.  The collection of isomorphism classes of rigidified $r$-torsion line bundles may be identified with the \'{e}tale cohomology group $\H^1_{\mathrm{et}}(X,\mu_r)$. 

Fix a pair of rigidified $r$-torsion line bundles $(\alpha,\beta)$ on $X$.  The cyclic algebra determined by $(\alpha,\beta)$ is then
\begin{equation}\label{cyclic:alg:defn}  
\mathcal{A}(\alpha,\beta) := \bigoplus_{0 \leq i,j \leq r-1} \alpha^{\otimes i} \otimes \beta^{\otimes j}.
\end{equation}
In \eqref{cyclic:alg:defn}, the algebra structure depends on $\zeta$ and is determined by
$$
\left( \alpha^{\otimes i} \otimes \beta^{\otimes j} \right) \otimes \left(\alpha^{\otimes k} \otimes \beta^{\otimes \ell} \right) \xrightarrow{\zeta^{jk}} \left( \alpha^{\otimes i} \otimes \alpha^{\otimes k} \right) \otimes \left( \beta^{\otimes j} \otimes \beta^{\otimes \ell} \right)  \rightarrow \alpha^{\otimes q} \otimes \beta^{\otimes s}\text{,}
$$
where 
$$
i + k \equiv q \mod r \text{ and } j + \ell \equiv s \mod r \text{,}
$$ 
for 
$0 \leq q, s \leq r- 1$.

The affine local structure of $\mathcal{A}(\alpha,\beta)$ is described by fixing affine open subsets $U \subseteq X$ which trivialize $\alpha$ and $\beta$.  In particular
$$
\alpha|_U \xrightarrow{\sim} \Osh_U e_{\alpha} \xrightarrow{\sim} \Osh_U = \Osh_X|_U
$$
and
$$
\beta|_U \xrightarrow{\sim} \Osh_U e_{\beta} \xrightarrow{\sim} \Osh_U = \Osh_X|_U. 
$$
Here $e_\alpha$ and $e_\beta$ are local generators for $\alpha|_U$ and $\beta|_U$ respectively.  Let 
$$\phi \colon \Osh_X \xrightarrow{\sim} \alpha^{\otimes r}$$ 
and 
$$
\psi \colon \Osh_X \xrightarrow{\sim} \beta^{\otimes r}
$$ 
denote the fixed rigidifications  of $\alpha^{\otimes r}$ and $\beta^{\otimes r}$.  There exists units $a,b \in \Gamma(U,\Osh_U)$ with the properties that
$$
a \cdot \phi(1)|_U = e_\alpha^{\otimes r} 
$$
and 
$$
b \cdot \psi(1)|_U = e_\beta^{\otimes r} \text{.}
$$
It follows that $\mathcal{A}(\alpha,\beta)|_U$ is isomorphic, as an $\Osh_U$-algebra, to the $\Osh_U$-algebra which is generated by the elements $\mathbf{e}_\alpha$ and $\mathbf{e}_\beta$ subject to the relations that
$$
\mathbf{e}^r_\alpha = a \text{, } \mathbf{e}^r_\beta = b \text{ and } 
\mathbf{e}_\alpha \cdot \mathbf{e}_\beta = \zeta \mathbf{e}_\beta \cdot \mathbf{e}_\alpha
\text{.}
$$
Note that $\mathcal{A}(\alpha,\beta)$ is an \emph{Azumaya algebra} of rank $r^2$ over $A$.  

As in \cite{Kaji:1988}, we consider a concept of \emph{Hilbert symbol} for a pair of $r$-torsion line bundles.

\begin{defn}[{\cite{Kaji:1988}}]\label{Hilbert:Symbol:Defn}  Let $\{\alpha,\beta \}_r$ denote the class of $\mathcal{A}(\alpha,\beta)$ in 
$$\Br(X)[r] = \H^2_{\mathrm{et}}(X, \G_m)[r].$$
We say that $\{\alpha,\beta \}_r$ is the \emph{Hilbert symbol} of $\alpha$ and $\beta$.
\end{defn}

The vanishing of such Hilbert symbols $\{\alpha, \beta \}_r$ detects when the cyclic algebra $\mathcal{A}(\alpha,\beta)$ has trivial class in the Brauer group. 

\begin{proposition}[\cite{Kaji:1988}]
Let $(\alpha,\beta)$ be a fixed pair of rigidified $r$-torsion line bundles on $X$.  The following two assertions are equivalent.
\begin{enumerate}
\item[(i)]{The cyclic algebra $\mathcal{A}(\alpha, \beta)$ has the form 
$$\mathcal{A}(\alpha,\beta) \simeq \mathcal{E}nd(\mathcal{E})\text{,}$$ 
for some vector bundle $\mathcal{E}$ on $X$.
}
\item[(ii)]{
The Hilbert symbol $\{\alpha,\beta\}_r$ determined by $(\alpha,\beta)$ is trivial 
$$\{\alpha,\beta\}_r = 0.$$
}
\end{enumerate}
\end{proposition}
\begin{proof}
This can be checked locally as noted in \cite[Proposition 1.5]{Kaji:1988}.
\end{proof}

By analogy with the Hilbert tame symbols that arise in the Milnor K-theory of fields, \cite[Chapter 7]{Gille:Szamuley:2006}, the Hilbert symbols, in the sense of Definition \ref{Hilbert:Symbol:Defn}, exhibit the following properties.
 
\begin{proposition}[\cite{Kaji:1988}]
Let $\alpha$, $\beta$ and $\gamma$ be a collection of rigidified $r$-torsion line bundles on $X$.  The following properties hold true
\begin{enumerate}
\item[(i)]{
$\{\alpha \otimes \beta, \gamma \}_r = \{\alpha,\gamma\}_r + \{\beta,\gamma \}_r$;
}
\item[(ii)]{
$\{\alpha,\beta \otimes \gamma \}_r = \{\alpha, \beta\}_r + \{\alpha,\gamma\}_r$; and 
}
\item[(iii)]{
$
\{\alpha,\beta\}_r + \{\beta,\alpha\}_r = 0.
$
}
\end{enumerate}
\end{proposition}

\begin{proof}
This can also be checked locally via a cup-product calculation, \cite[Corollary 2.6]{Kaji:1988}.
\end{proof}

The first main result for $r$-torsion in the Brauer group of $X$ is stated in the following way.

\begin{theorem}[{\cite[Chapter IV]{Milne:Etale}}]\label{Ab:var:Brauer:presentation}
Let $X$ be a nonsingular projective variety over an algebraically closed field $\kk$ and fix a positive integer $r$ which is not divisible by the characteristic of $\kk$.  In this context, the $r$-torsion part of the Brauer group $\Br(X)$ is generated by the Hilbert symbols.  Furthermore, the group $\Br(X)[r]$ sits in an exact sequence
$$
0 \rightarrow \Pic(X) / r \Pic(X) \rightarrow \H^2_{\mathrm{et}}(X,\mu_r)  \rightarrow \Br(X)[r] \rightarrow 0.
$$
\end{theorem}
\begin{proof}
The key point is that this exact sequence is induced by the Kummer sequence
$$
1 \rightarrow \mu_r \rightarrow \G_m \xrightarrow{\cdot r} \G_m \rightarrow 1.
$$
We refer to \cite[Chapter IV]{Milne:Etale} for more details.
\end{proof}

For the case of Abelian varieties, Theorem \ref{Ab:var:Brauer:presentation} takes a more explicit form.

\begin{corollary}[\cite{Berkovich:1972}, \cite{Kaji:1988}]\label{Ab:var:Brauer:presentation:rank}  Let $A$ be a $g$-dimensional Abelian variety over an algebraically closed field $\kk$.  Let $\rho$ be the Picard number of $A$.  In this context, the $r$-torsion part of $\Br(A)[r]$, for $r$ not divisible by the characteristic of $\kk$, is a free $\ZZ / r \ZZ$-module of rank $\binom{2g}{2} - \rho$.
\end{corollary}
\begin{proof}
Since
$$\H^2_{\mathrm{et}}(A,\mu_r) = \Hom\left(\bigwedge^2 A[r], \mu_r\right)\text{,}$$
the conclusion of Corollary \ref{Ab:var:Brauer:presentation:rank} follows from Theorem \ref{Ab:var:Brauer:presentation}.
\end{proof}

\begin{example}\label{E1:times:E2:examples}
As an illustration of Corollary \ref{Ab:var:Brauer:presentation:rank}, it is helpful to consider the case of a complex Abelian variety $A$ that is the product of elliptic curves.  By \cite{Skorobogatov:Zarhin:2012}, see also \cite{Skorobogatov:Zarhin:2017}, the discussion that follows also allows for a description of the Brauer group of $\operatorname{Kum}(A)$, the nonsingular Kummer variety that is determined by $A$.

First, let $\mathfrak{h}$ denote the upper half plane and consider the case that 
$$A \simeq E_1 \times E_2$$ for elliptic curves 
$$E_i \simeq \mathbb{C} / \Lambda_i\text{,}$$ where 
$$\Lambda_i = \ZZ \oplus \tau_i \ZZ$$
 and $\tau_i \in \mathfrak{h}$, $i=1,2$.  
 Then
\begin{align*}
\operatorname{Br}(A) & \simeq \Hom_{\ZZ}(T(A),\QQ/\ZZ) \\
& \simeq \Hom_{\ZZ}(T(A),\ZZ)\otimes \QQ/\ZZ \\
& \simeq \H^2(A,\ZZ)/\operatorname{NS}(A) \otimes \QQ / \ZZ
\end{align*}
where 
$$
T(A) \simeq \operatorname{NS}(A)^{\perp}
$$
is the transcendental lattice.

In order to determine the N\'{e}ron-Severi group
$$
\operatorname{NS}(A) \simeq \H^{1,1}(A,\ZZ) \text{,}
$$
let $x_1,\dots,x_4$ be real coordinates corresponding to the dual basis of
\begin{align*}
\H_1(X,\ZZ)  & = \H_0(E_2,\ZZ)  \otimes \H_1(E_1,\ZZ) \oplus \H_1(E_2,\ZZ) \otimes \H_0(E_1,\ZZ) \\
& = \Lambda_1 \oplus \Lambda_2 \text{.}
\end{align*}
Let 
$$\omega_1 = dx_1 + \tau_1 dx_3$$ 
and 
$$\omega_2 = d x_3 + \tau_2 d x_4 \text{.}$$  
Then 
$$\overline{\omega}_1 = dx_1 + \overline{\tau}_1 d x_2$$ 
and
$$\overline{\omega}_2 = d x_3 + \overline{\tau}_2 d x_4\text{;}$$ 
the ordered bases
$$
dx_1 \wedge dx_2, dx_1 \wedge dx_3, dx_1 \wedge d x_4, dx_2 \wedge dx_3, dx_2 \wedge dx_4,dx_3\wedge dx_4
$$
and
$$
\omega_1 \wedge \omega_2,\omega_1 \wedge \overline{\omega}_1,\omega_1 \wedge \overline{\omega}_2,\omega_2 \wedge \overline{\omega}_1 ,\omega_2 \wedge \overline{\omega}_2,\overline{\omega}_1\wedge\overline{\omega}_2 \text{, }
$$
for $\H^2(A,\CC)$, are related by the transition matrix
$$
\left( \begin{matrix}
0 & \frac{\overline{\tau}_1 \ \overline{\tau}_2}{(\overline{\tau}_1-\tau_1)(\overline{\tau}_2-\tau_2)} & \frac{-\overline{\tau}_1}{(\overline{\tau}_1-\tau_1)(\overline{\tau}_2-\tau_2)} & \frac{-\overline{\tau}_2}{(\overline{\tau}_1-\tau_1)(\overline{\tau}_2-\tau_2)} & \frac{1}{(\overline{\tau}_1-\tau_1)(\overline{\tau}_2-\tau_2)} & 0 \\
\frac{\overline{\tau}_1-\tau_1}{(\overline{\tau}_1-\tau_1)^2} & 0 & 0 & 0 & 0 &0 \\
0 & \frac{-\overline{\tau}_1\tau_2}{(\overline{\tau}_1-\tau_1)(\overline{\tau}_2-\tau_2)} & \frac{\overline{\tau}_1}{(\overline{\tau}_1-\tau_1)(\overline{\tau}_2-\tau_2)} &\frac{\overline{\tau}_2}{{(\overline{\tau}_1-\tau_1)(\overline{\tau}_2-\tau_2)}} & \frac{-1}{(\overline{\tau}_1-\tau_1)(\overline{\tau}_2-\tau_2)} & 0 \\
0&\frac{\tau_1 \overline{\tau}_2}{(\overline{\tau}_1-\tau_1)(\overline{\tau}_2-\tau_2)} & \frac{-\tau_1}{(\overline{\tau}_1-\tau_1)(\overline{\tau_2}-\tau_2)} & \frac{-\overline{\tau_2}}{(\overline{\tau}_1-\tau_1)(\overline{\tau}_2-\tau_2)} & \frac{1}{(\overline{\tau}_1-\tau_1)(\overline{\tau}_2-\tau_2)} & 0 \\
0 & 0&0 &0 & 0& \frac{\overline{\tau}_2-\tau_2}{(\overline{\tau}_2-\tau_2)^2} \\
0 & \frac{\tau_1 \tau_2}{(\overline{\tau}_1-\tau_1)(\overline{\tau}_2-\tau_2)} & \frac{-\tau_1}{(\overline{\tau}_1-\tau_1)(\overline{\tau}_2-\tau_2)} & \frac{-\tau_2}{(\overline{\tau}_1-\tau_1)(\overline{\tau}_2-\tau_2)} & \frac{1}{(\overline{\tau}_1-\tau_1)(\overline{\tau}_2-\tau_2)} & 0 
\end{matrix} \right) 
$$

Among other things, this matrix allows for calculation of the Picard number $\rho(A)$. For instance, the N\'{e}ron-Severi group $\H^{1,1}(A,\ZZ)$ consists exactly of those forms
$$
a_{12}dx_1\wedge dx_2 + a_{13}dx_1\wedge dx_3 + a_{14}dx_1\wedge dx_4 + a_{23}dx_2\wedge dx_3 + a_{24} dx_2 \wedge dx_4 + a_{34}dx_3 \wedge dx_4
$$
where 
$a_{12},a_{13},a_{14},a_{23},a_{24},a_{34} \in \ZZ$ 
and the collection of integers $a_{13}, a_{14}, a_{23}, a_{24}$  satisfies the equation
\begin{equation}\label{Hodge:bilinear:eqns}
b_{13}\tau_1 \tau_2 - b_{14}\tau_1 - b_{23}\tau_2+b_{24} = 0 \text{,}
\end{equation}
for unknown integers $b_{13}$, $b_{14}$, $b_{23}$ and $b_{24}$.
In more explicit terms, the following assertions hold true.
\begin{itemize}
\item{If $\rho(A) = 2$, then $\H^{1,1}(A,\ZZ)$ is spanned by the forms
\begin{enumerate}
\item[(i)]{
$dx_1 \wedge dx_2$; and
}
\item[(ii)]{ $dx_3 \wedge dx_4$.
}
\end{enumerate}
}
\item{ 
If $\rho(A) = 3$, then $\H^{1,1}(A,\ZZ)$ is spanned by the forms
\begin{enumerate}
\item[(i)]{
$dx_1 \wedge dx_2$;
}
\item[(ii)]{$dx_3 \wedge dx_4$; and
}
\item[(iii)]{
$
a_{13} dx_1 \wedge dx_3 + a_{14} dx_1 \wedge dx_4 + a_{23} dx_2 \wedge dx_3 + a_{24} dx_2 \wedge dx_4 \text{,}
$
where $a_{13}, a_{14}, a_{23}, a_{24} \in \ZZ$ form a solution of the equation \eqref{Hodge:bilinear:eqns}.
}
\end{enumerate}
}
\item{
If $\rho(A) = 4$, then $\H^{1,1}(A,\ZZ)$ is spanned by the forms
\begin{enumerate}
\item[(i)]{$dx_1\wedge dx_2$;
}
\item[(ii)]{
$dx_3\wedge dx_4$;
}
\item[(iii)]{
$
a_{13} dx_1 \wedge dx_3 + a_{14} dx_1 \wedge dx_4 + a_{23} dx_2 \wedge dx_3 + a_{24} dx_2 \wedge dx_4 \text{;}
$
and
}
\item[(iv)]{
$
a_{13}' dx_1 \wedge dx_3 + a_{14}' dx_1 \wedge dx_4 + a_{23}' dx_2 \wedge dx_3 + a_{24}' dx_2 \wedge dx_4 \text{.}
$
}
\end{enumerate}
Here, in (iii) and (iv) above, the collection integers 
$$a_{13}, a_{14}, a_{23}, a_{24}, a_{13}', a_{14}', a_{23}', a_{24}' \in \ZZ$$ 
form two solutions to the equation \eqref{Hodge:bilinear:eqns}.
}
\end{itemize}
In particular, the Picard number $\rho(A)$ is described as
$$
\rho(A) = 
\begin{cases}
2 & \text{if and only if $E_1$ and $E_2$ are not isogenous} \\
3 & \text{if and only if $E_1$ and $E_2$ are isogenous and do not have CM} \\
4 & \text{ if and only if $E_1$ and $E_2$ are isogenous and have CM.}
\end{cases}
$$

Turing to the case that 
$$A \simeq E_1\times\dots\times E_g$$ 
is the $g$-fold product of a collection of elliptic curves 
$$E_i \simeq \CC / \Lambda_i\text{,}$$ 
$\tau_i \in \mathfrak{h}$, $i=1,\dots,g$, the above discussion combined with the Theorem of the Cube, \cite{Mum:v1}, allows for calculation of the Picard number $\rho(A)$.  For example, the Picard number $\rho(A)$ is described as
\begin{equation}\label{picard:number:g:fold:elliptic:curve:product}
\rho(A) = g + \sum\limits_{i=1}^{g-1} \sum\limits_{j=i+1}^g \rho_{\mathrm{new}}(\E_i \times \E_j)\text{;} 
\end{equation} 
here
$$\rho_{\mathrm{new}}(\E_i \times \E_j) = 
\begin{cases} 
0 & \text{ if and only if $\E_i$ and $\E_j$ are not isogenous} \\
1 & \text{ if and only if $\E_i$ and $\E_j$ are isogenous and do not have CM} \\
2 & \text{if and only if $\E_i$ and $\E_j$ are isogenous and have CM.}
\end{cases} $$

Finally, turning to the Brauer group of such a $g$-fold product of elliptic curves $A$, applying Corollary \ref{Ab:var:Brauer:presentation:rank}, it follows that $\operatorname{Br}(A)[r]$ is a free $\ZZ/r\ZZ$-module of rank $\binom{2g}{2} - \rho(A)$, where $\rho(A)$ is described in \eqref{picard:number:g:fold:elliptic:curve:product}. 
\end{example}

\section{Nondegenerate theta groups and the Schr\"{o}dinger Representations}\label{weight:one:theta}

In this section, for later use, we discuss the weight $1$ representation theory of nondegenerate theta groups.  It is a special case of the general theory given in \cite[Chapter 5]{Moret-Bailly}, which  generalizes previous work of Mumford \cite[Section 1]{MumI} and \cite[Section 6]{MumII}.  For our purposes, we let $S := \spec(\kk)$ for $\kk$ an algebraically closed base field.

Let $\pi \colon \K \rightarrow S$ be finite commutative group scheme, locally free of rank $r^2$, and fix a \emph{nondegenerate theta group} over $\K$
\begin{equation}\label{non:deg:theta:eqn1}
1 \rightarrow \G_m \rightarrow \G \rightarrow \K \rightarrow 1.
\end{equation}
Equivalently, \eqref{non:deg:theta:eqn1} is a central extension and the alternating form
\begin{equation}\label{weight:decomp:eqn2}
e \colon \K \times \K \rightarrow \G_m,
\end{equation}
which is induced by the commutator of $\G$,
is a perfect (nondegenerate) pairing.   

Let $M$ be a \emph{$\G$-module}.  Restricting the $\G$-action to $\G_m$, yields the weight decomposition 
\begin{equation}\label{weight:decomp:eqn3}
M = \bigoplus_{i \in \ZZ} M^{(i)}.
\end{equation}
In \eqref{weight:decomp:eqn3}, $M^{(i)}$ is the largest submodule of $M$ where $\G_m$ acts by the character $\lambda \mapsto \lambda^i$.  The $\G$-module $M$ is said to be of \emph{weight} $i \in \ZZ$ in case that $M = M^{(i)}$.

Fix a line bundle $\mathcal{L}$ on $\K$ so that
$$
\G = \operatorname{Isom}_{\Osh_\K-\mathrm{mod}}(\Osh_\K,\mathcal{L}).
$$
Put
$$
A^{(i)}(\G) = \pi_*(\mathcal{L}^{\otimes {- i}})
$$
and consider the natural action of $\G$ on its coordinate ring
\begin{equation}\label{weight:decomp:eqn4}
A(\G) = \bigoplus_{i \in \ZZ} A^{(i)}(\G).
\end{equation}
Note that each $A^{(i)}(\G)$ is a locally free $\Osh_S$-module of rank $r^2$.

The natural action of $\G \times \G$ on $A(\G)$ respects the decomposition \eqref{weight:decomp:eqn4}.  Let $E^{(1)}$ be the $(\G \times \G)$-module determined by $A^{(1)}(\G)$ and $E_+^{(1)}$, respectively $E_-^{(1)}$, the $\G$-module obtained by restricting to $\{1 \} \times \G$, respectively $\G \times \{1\}$, the action of $\G \times \G$ on $E^{(1)}$.  Then $E^{(1)}_+$ has weight $1$ whereas $E^{(1)}_-$ has weight $-1$.

\noindent
\begin{lemma}  
In the above notation, $E^{(1)}$ is an irreducible $\G \times \G$-module.  In particular, each $\G \times \G$-submodule of $E^{(1)}$ is of the form $I \cdot E^{(1)}$ where $I$ is an ideal of $\Osh_S$.
\end{lemma}
\begin{proof}
This is a special case of \cite[Lemma 2.3, page 111]{Moret-Bailly}.
\end{proof}

Let $\mathrm{Rep}^1(\G)$ be the Abelian category of weight one representations.  Fix $V \in \mathrm{Rep}^1(\G)$ and suppose that $V$ is locally free of rank $r$ as an $\Osh_S$-module.  Let $V^\vee$ be the dual of $V$.  In this context, there exists a natural homomorphism
\begin{equation}\label{theta:natural:homomorphism:eqn1}
V \otimes_{\Osh_S} V^\vee \rightarrow A^{(1)}(\G) = E^{(1)}
\end{equation}
which is defined by
$$
v \otimes v^\vee \in V \otimes V^\vee \mapsto (g \mapsto \langle gv, v^\vee \rangle).
$$
This is a $\G \times \G$ morphism with $\G \times \{1 \}$ acting by the given $\G$-representation and $\{1\} \times \G$ acting by the natural dual action.

The weight $1$ representation theory of $\G$ is described as follows.

\begin{theorem}[{\cite{Moret-Bailly}}]
The following assertions hold true.
\begin{enumerate}
\item[(i)]{
The natural homomorphism \eqref{theta:natural:homomorphism:eqn1} is an isomorphism.
}
\item[(ii)]{
If $M$ is a $\G$-module with weight $1$, then the natural $\G$-morphism
$$ 
V \otimes \Hom_\G(V,M)^{\mathrm{triv}} \rightarrow M
$$
defined by
$$
v \otimes u \mapsto u(v)
$$
is an isomorphism.
}
\item[(iii)]{
If $N$ is a quasi-coherent $\Osh_S$-module, then the natural $\G$-morphism
$$
N \mapsto \Hom_\G(V,V\otimes N^{\mathrm{triv}})
$$
defined by 
$$
x \mapsto (v \mapsto v \otimes x)
$$
is an isomorphism.
}
\end{enumerate}
\end{theorem}
\begin{proof}
This is a special case of \cite[Theorem 2.4.2, page 112]{Moret-Bailly}.
\end{proof}

The \emph{regular weight $1$}, or \emph{Schr\"{o}dinger Representations} of $\G$ are determined by level and Lagrangian subgroups.

\begin{defn}
A \emph{level subgroup} of $\G$ is a subgroup scheme $\H \subseteq \G$, finite locally free over $S$ and such that $\H \bigcap \G_m = \{1 \}$.  A level subgroup is \emph{Lagrangian} in case that it has rank $r$ over $S$.
\end{defn}

The \emph{Schr\"{o}dinger Representations} are constructed as follows.

\begin{proposition}\label{level:subgroup:weight1:rep}
Let $\H \subseteq \G$ be a level subgroup scheme of rank $r'$ and put
$$
V = \left(E_-^{(1)}\right)^{\H} = \left(E^{(1)}\right)^{\{1\} \times \H}.
$$
Then $V$ is a $\G$-submodule of $E^{(1)}_+$, locally free of rank $r^2 / r'$ as a $\Osh_S$-module.  In particular, if $\H$ is Lagrangian, then $V$ is locally free of rank $r$.
\end{proposition}
\begin{proof}
This is a special case of \cite[Proposition 2.5.2, page 115]{Moret-Bailly}.
\end{proof}

\begin{remark}  When $r$ is not divisible by the characteristic of $\kk$, then Lagrangian subgroups exist \cite{MumI}.   Further, in this context, the nature of higher weight representations of the nondegenerate theta groups is determined in \cite{Grieve:theta}. 
\end{remark}

\section{Proof of Theorem \ref{Brion:Br:thm}}\label{proof:Brion:thm}

The purpose of this section is to prove Theorem \ref{Brion:Br:thm}.  Our presentation follows Brion's suggestion \cite[Remark 3.13]{Bri}.  For completeness, we restate Theorem \ref{Brion:Br:thm} as Theorem \ref{Brauer:AH} below.  For the case of complex Abelian varieties, it refines \cite[Theorem 1]{Elencwajg:Narasimhan:1983}.

\begin{theorem}\label{Brauer:AH}
Let $A$ be an Abelian variety over an algebraically closed field $\kk$.  Fix a positive integer $r$ which is not divisible by the characteristic of $\kk$.  Let $\operatorname{Br}(A)[r]$ be the $r$-torsion subgroup of the Brauer group of $A$.  Then each class in $\operatorname{Br}(A)[r]$ may be represented by a homogeneous irreducible projective bundle.
\end{theorem}

\begin{proof}[Proof of Theorem \ref{Brion:Br:thm} and \ref{Brauer:AH}]
 First, recall, as in \cite[page  227]{Mum:v1}, that if $M$ is a line bundle on $A$ and if
$$
\K(M) := \{ x \in A : \tau_x^* M \simeq M \} \text{,}
$$
then the commutator of the theta group
$$
\G(M) := \{(\phi,x) : x \in \K(M) \text{ and }  \phi \colon \tau_x^* M \xrightarrow{\sim} M \} \text{,}
$$
determines an alternating bilinear form
$$
e^M \colon \K(M) \times \K(M) \rightarrow \G_m \text{.}
$$
In what follows, if
$$
H \subseteq \K(M)
$$
is a subgroup scheme, then let $e^M|_{H}$ denote the induced alternating form
$$
e^M|_H \colon H \times H \rightarrow \G_m \text{.}
$$

Now, as in Theorem \ref{Ab:var:Brauer:presentation} and Corollary \ref{Ab:var:Brauer:presentation:rank}, the Kummer sequence
$$
1 \rightarrow \mu_r \rightarrow \mathbb{G}_m \xrightarrow{\cdot r} \mathbb{G}_m \rightarrow 1
$$ 
induces an exact sequence
$$
0 \rightarrow \NS(A) / r \NS(A) \xrightarrow{\phi} \H^2_{\mathrm{et}}(A,\mu_r) = \Hom_{\ZZ/r\ZZ}\left( \bigwedge^2 A[r],\mu_r \right ) \xrightarrow{\psi} \H^2_{\mathrm{et}}(A,\mathbb{G}_m)[r] \rightarrow 0.
$$
The map $\phi$ sends the class of a line bundle $L$ to the form $e^{rL}|_{A[r]}$ and the map $\psi$ sends a form $e$ to the class of the Azumaya algebra
\begin{equation}\label{homog:irred:1}
\mathcal{A} := \bigoplus_{\substack{ \alpha \in A^t[r] \\ \sigma \in A[r] }} \alpha \cdot \mathbf{e}_{\sigma}.
\end{equation}
In \eqref{homog:irred:1}, we identify points $\alpha \in A^t[r]$ with the $r$-torsion line bundles $\alpha$ that they determine.  Let 
$$e_r(\cdot,\cdot) \colon A[r] \times A^t[r] \rightarrow \mu_r$$ 
be the canonical pairing.  Then the multiplication in \eqref{homog:irred:1} is defined, for local sections $s_\alpha$ of $\alpha$, and $s_\beta$ of $\beta$, by the condition that
$$
\left(s_\alpha \cdot \mathbf{e}_\sigma \right) \cdot \left( s_\beta \cdot \mathbf{e}_\tau \right) = (e_r(\beta,\sigma) \cdot a_{\sigma,\tau} \cdot s_\alpha \cdot s_\beta)  \cdot \mathbf{e}_{\sigma + \tau}. 
$$
Here 
$$
\{a_{\sigma,\tau} \} \in Z^2(A[r],\G_m)
$$
is a $2$-cocycle that has the property that
$$
e(\sigma,\tau) = a_{\sigma,\tau} \cdot a_{\tau,\sigma}^{-1} \text{.}
$$

The algebra \eqref{homog:irred:1} has rank $r^{2g}$ and is free as a left module over the maximal \'{e}tale subalgebra
$$
\mathcal{L} := \bigoplus_{\alpha  \in A^t[r]} \alpha \cdot \mathbf{e}_0.
$$
By \cite[Corollary 5.5]{Grothendieck:Brauer:I}, it follows that
\begin{equation}\label{homog:irred:2}
[r]_A^*\mathcal{A} \simeq M_{r^{2g}}(\Osh_A).
\end{equation}
Let $\P$ be the $\PP^{r^g-1}$-bundle that corresponds to the Azumaya algebra \eqref{homog:irred:1}.  Then, following the terminology of \cite[page 2493]{Bri}, $\P$ is trivialized by $[r]_A$.  Indeed, this is implied by \eqref{homog:irred:2}.

The above discussion implies that each class in $\Br(A)[r]$ is represented by a homogeneous $\PP^{r^g-1}$-bundle which is trivialized by $[r]_A$.  Write the fixed choice of such a bundle as
\begin{equation}\label{homog:irred:3}
\P \simeq A \times^{A[r]} \PP^{r^g-1} \text{, }
\end{equation}
where the action of $A[r]$ on $\PP^{r^g-1}$ is achieved via some projective representation 
$$
\rho \colon A[r] \rightarrow \PGL_{r^g} 
$$
as in \cite[Proposition 3.9 (i)]{Bri}.  

Let
\begin{equation}\label{homog:irred:6}
e_{\P,r}(\cdot,\cdot) = e(\cdot,\cdot) \colon A[r] \times A[r] \rightarrow \mu_r
\end{equation}
be the alternating bilinear form that is determined by this trivialization, \cite[pages 2494, 2495]{Bri},
and let 
$$
A[r]^\perp = \{x \in A[r] : e(x,y) = 1 \text{ for all } y \in A[r]\}
$$
be the orthogonal complement of $A[r]$ with respect to $e$.  
Then
$$
[A[r]:A[r]^\perp] = d^2,
$$
for $d$ the \emph{homogeneous index} of the bundle $\P$ \cite[page 2487]{Bri}.  It is the minimal rank of a homogeneous subbundle.

Let $S \subseteq \PP^{r^g-1}$ be a linear subspace which is $A[r]$-stable and which is minimal for this property.  Then
$$
\P_1 := A \times^{A[r]} \PP^{d-1}
$$
is a homogeneous irreducible $\PP^{d-1}$-bundle.  In terms of the classification of Theorem \ref{homog:characterize} (i), 
it is obtained via the anti-affine extension
$$
1 \rightarrow A[r]/A[r]^\perp \rightarrow A/A[r]^\perp \rightarrow A/A[r] = A  \rightarrow 1
$$
together with the induced faithful representation
$$
\rho_1 \colon A[r] / A[r]^\perp \rightarrow \PGL_{{r^g},d} \rightarrow \PGL_d.
$$
Here, $\PGL_{r,d} \subseteq \PGL_{r^g-1}$ is the maximal parabolic subgroup that stabilizes 
$$\PP^{d-1} \simeq S \subseteq \PP^{r^g-1}\text{.}$$

Finally, it remains to show that $\P_1$ and $\P$ have the same class in $\Br(A)[r]$.  To this end, $\P_1 \cdot \P_1^*$ is the projectivization of a vector bundle.  On the other hand, $\P_1\cdot  \P_1^*$ is a subbundle of $\P \cdot \P_1^*$.  It follows from \cite[Lemma 3.12]{Bri}, that $\P \cdot \P_1^*$ is the projectivization of a vector bundle too.
\end{proof}

\section{Homogeneous irreducible Severi-Brauer varieties over Abelian varieties}

Our  main focus here is the study of \emph{homogeneous Severi-Brauer varieties} over an Abelian variety $A$.  Recall, that a $\PP^{n-1}$-bundle, $p \nmid n$, over $A$ is \emph{homogeneous} if it is isomorphic, as a bundle over $A$, to its pullback under all translations.  
A homogeneous $\PP^{n-1}$-bundle  is \emph{irreducible} if it contains no non-trivial proper homogeneous subbundle \cite[Proposition 3.5]{Bri}.

We describe the basic method of \cite{Bri} for constructing \emph{homogeneous irreducible} Severi-Brauer varieties over $A$.  It is implicit in the proof of Corollary \ref{homog:irred:bundles}.  In particular, each finite subgroup scheme $\K \subseteq A[n]$, with $\# \K = n^2$ and $p \nmid n$, together with a faithful irreducible projective representation 
\begin{equation}\label{bundle:theta:rep}
\rho \colon \K \rightarrow \PGL_n(\kk),
\end{equation}
determines a homogeneous $\PP^{n-1}$-bundle over $A$.  Such bundles may be described as
\begin{equation}\label{homog:irred:construct}
 \P :=  (A^t/\K^t)^t \times^\K \PP^{n-1}.
 \end{equation}
 In \eqref{homog:irred:construct}, $\K$ acts on $\PP^{n-1}$ via the representation \eqref{bundle:theta:rep} and on $(A^t/\K^t)^t$ via translation.  Here, we use the fact that 
  \begin{equation}\label{dual:eqn1}
 \K^{t t} = \K
 \end{equation} 
 is the kernel of the isogeny
 \begin{equation}\label{dual:eqn2}
 A = A^{tt} \leftarrow \left( A^t / \K^t \right)^t
 \end{equation}
 that is dual to 
 \begin{equation}\label{dual:eqn3}
 A^t \rightarrow A^t / \K^t \text{,}
 \end{equation}
 \cite[Theorem 11.1]{Milne:1986}.
 
As is a consequence of Proposition \ref{level:subgroup:weight1:rep}, such homogeneous $\PP^{n-1}$-bundles are \emph{irreducible}.  For later use, we record the fact that the theory of the Schr\"{o}dinger representations can be used to construct homogeneous irreducible $\PP^{n-1}$-bundles over $A$.

\begin{proposition}\label{Schrodinger:Azumaya}  Let $A$ be an Abelian variety over an algebraically closed field $\kk$.  
Let $\K \subseteq A[n]$ be a finite subgroup of rank $n^2$, for $n$ not divisible by the characteristic of $\kk$, and consider a nondegenerate theta group over $\K$
$$ 1 \rightarrow \G_m \rightarrow \G \rightarrow \K \rightarrow 1\text{,}$$
with Schr\"{o}dinger representation $(V, \rho)$ determined by a Lagrangian subgroup scheme $\H \subseteq \G$.  Then $(V, \rho)$ determines a homogeneous irreducible $\PP^{n-1}$-bundle over $A$.  Conversely, each homogeneous irreducible $\PP^{n-1}$-bundle over $A$ arises in this way.
\end{proposition}

\begin{proof}[Proof]
Consider the faithful projective projective representation
$$
\rho \colon \K \rightarrow \operatorname{PGL}_n
$$
that is determined by $(V,\rho)$.  Then, as in the above discussion, we obtain a $\PP^{n-1}$-bundle over $A$
$$ \P_{(V,\rho)} := (A^t/\K^t)^t \times^{\K} \PP^{n-1}.$$
Here the quotient is taken with respect to the diagonal action of $\K$.  In particular, $\K$ acts by
$$k \cdot(x,v) := (k + x, \rho(k)\cdot v).$$  
The fact that $\P_{(V, \rho)}$ is homogeneous is evident.  That $\P_{(V, \rho)}$ is irreducible follows from Proposition \ref{level:subgroup:weight1:rep}.  Finally, the converse follows via Corollary \ref{homog:irred:bundles} and Theorem \ref{Brauer:AH}.  \end{proof}

\section{Biextensions, cubic structures and theta groups}

In this section, we briefly describe some concepts which pertain to biextensions, cubic structures and theta groups.  More details can be found in \cite{Mumford:Formal:Groups}, \cite{Breen:1983}, and \cite{Moret-Bailly}.  In Section \ref{S-B-cubical-structures}, we use these concepts to study Severi-Brauer varieties and Azumaya algebras over Abelian varieties.  For the basic structure theory of commutative groups schemes, over a fixed algebraically closed base field $\kk$, we refer to \cite{Oort:1966}.  

Let $\K$ be a  commutative group scheme over $\kk$ and
$
\pi \colon L \rightarrow \K
$
 a $\G_m$-torsor.  Over $\K \times \K$, we let 
$$
 \pi \hat{*} \pi \colon L \hat{*} L \rightarrow \K \times \K
$$
denote the $\G_m$-torsor defined by the condition that its fiber over 
$$(x,y) \in \K \times \K$$ 
is described by
$$ L \hat{*} L |_{x,y} = L^{-1} |_x \otimes L^{-1}|_y \otimes L|_{x+y}.$$
This torsor carries a canonical symmetry
$$ 
\xi_{x,y} \colon L \hat{*} L|_{x,y} \xrightarrow{\sim} L \hat{*} L|_{y,x}.
$$
Over $\K \times \K \times \K$, let $\Theta(L)$ denote the $\G_m$-torsor which has fiber over 
$$(x,y,z) \in \K \times \K \times \K$$
given by
$$ 
\Theta(L)|_{x,y,z} = L|_x \otimes L|_y \otimes L|_z \otimes L^{-1}|_{x+y} \otimes L^{-1}|_{x+z} \otimes L^{-1}|_{y+z} \otimes L|_{x+y+z}.
$$

There is a concept of a \emph{symmetric biextension structure} for $L \hat{*} L$ as a $\G_m$-torsor over $\K \times \K$, \cite[Section 1]{Breen:1983}, \cite[Chapter I.2.5]{Moret-Bailly}, see also \cite[Section 2]{Mumford:Formal:Groups}.  We require two composition laws 
$$ 
+_1 \colon L \hat{*} L \times  L \hat{*} L \rightarrow L \hat{*} L
$$
and
$$ 
+_2 \colon L \hat{*} L \ \times  \ L \hat{*} L \rightarrow L \hat{*} L.
$$
These laws are required to have the properties that
\begin{enumerate}
\item[(i)]{
for each $x \in \K$, via $+_1$, the pullback
$ 
L \hat{*} L |_{x \times \K} \text{,} 
$ of $L \hat{*} L$ over $x \times \K$ is a commutative group scheme and the induced morphism 
$$
L \hat{*} L |_{x \times \K} \rightarrow \K
$$ 
is a surjective homomorphism with kernel equal to $\G_m$;
}
\item[(ii)]{
for each $y \in \K$, via $+_2$, the pullback
$
  L \hat{*} L |_{\K \times y} \text{,}
 $ 
 of $L \hat{*} L$ over $\K \times y$ is a commutative group scheme and the induced morphism 
 $$
 L \hat{*} L|_{\K \times y} \rightarrow \K
 $$ 
 is a
 surjective homomorphism with kernel equal to $\G_m$; and
}
\item[(iii)]{
if $x,y,u,v \in L \hat{*} L$ have respective image in $\K \times \K$ given by 
\begin{itemize}
\item{
$ \pi \hat{*} \pi(x) = (b_1,c_1)$;
}
\item{ 
$\pi \hat{*} \pi(y) = (b_1, c_2)$;
} 
\item{
$\pi \hat{*} \pi(u) = (b_2,c_1)$; and
} 
\item{
$\pi \hat{*} \pi(v) = (b_2, c_2)$,
}
\end{itemize}
then:
$$ (x +_1 y ) +_2 (u +_1 v) = (x +_2 u) +_1 (y +_2 v).$$
}
\end{enumerate}

There is a concept of \emph{cubic structure} for $L$ as a $\G_m$-torsor over $\K$, \cite[Section 2]{Breen:1983}, \cite[Chapter I.2.4]{Moret-Bailly}.  This is a section $t$ of $\Theta(L)$, as a $\G_m$-torsor over $\K \times \K \times \K$, which 
 is invariant under the natural action of the symmetric group on three letters and which is also a two cocycle in each pair of variables.  We say that the pair $(L,t)$ is a \emph{cubic torsor}.  There is an evident concept of morphisms between cubic torsors.  Let $\mathcal{C}ub(\K,\G_m)$ denote the category of cubic torsors over $\K$.

The following remark clarifies the relation amongst cubic and symmetric biextension structures.

\begin{proposition}[{\cite[Proposition 2.5.4]{Moret-Bailly}}]\label{cubic:symmetric:biextension:equivalence}  Let $\K$ be a  commutative group scheme over an algebraically closed based field $\kk$.  Let
$
\pi \colon L \rightarrow \K
$
 be a $\G_m$-torsor.
To give a symmetric biextension structure on the $\G_m$-torsor  $L \hat{*} L$, is equivalent to giving a cubic structure on $L$ as a $\G_m$-torsor over $\K$.
\end{proposition}

\begin{proof}
This is a special case of {\cite[Proposition 2.5.4]{Moret-Bailly}}, see also \cite[Section 2]{Breen:1983}.   Nevertheless, we include a proof here.  Specifically,  we explain how cubic structures on a $\G_m$-torsor $L$ are related to symmetric biextension structures on $L \hat{*} L$.  Important to that discussion is the following \emph{canonical commutative diagram}
$$
\xymatrix{
\Theta_3(L) = \Theta(L) \ar[d]^-{x_1}_-{\wr}  \ar[dr]_-{x_2}^{\sim} \\ 
(m \times 1)^* L \hat{*} L \otimes p_{13}^* (L \hat{*} L^{-1}) \otimes p_{23}^*( L \hat{*} L^{-1}) \ar[r]^-{\sim} & (1 \times m)^* L \hat{*} L \otimes p_{12}^* (L \hat{*} L)^{-1} \otimes p_{23}^* (L \hat{*} L)^{-1} \text{,}
}
$$
\cite[Equation 2.1.10]{Breen:1983}.
In the above diagram, $m$ denotes multiplication in the group law, $p_{ij}$ denotes projection onto the $i,j$-factor and  $x_i$, for $i = 1,2$, denotes the evident natural isomorphism.  These morphisms $x_i$ are induced, respectively, by the natural trivialization morphisms 
$\Osh_\K \xrightarrow{\sim} L^{-1}|_{x_1} \otimes L|_{x_1}$ 
and 
$\Osh_\K \xrightarrow{\sim} L^{-1}|_{x_2} \otimes L|_{x_2}\text{.}$

In particular, to begin with, suppose given a cubic structure $t$ on $\Theta(L)$.  We then deduce, using the above canonical commutative diagram, that $t$ determines a biextension structure on $L \hat{*} L$.  This biextension structure is compatible with the canonical symmetry
$$ \xi_{x,y} \colon L \hat{*} L|_{x,y} \xrightarrow{\sim} L \hat{*} L|_{y,x}.$$
Conversely, suppose given a symmetric biextension structure on the $\G_m$-torsor $L \hat{*} L$.  Especially, we are given sections $s_1$ and $s_2$ of, the respective, $\G_m$-torsors over $\K^3$
$$ (m \times 1)^* L \hat{*} L \otimes p_{13}^* L \hat{*} L^{-1} \otimes p_{23}^* L \hat{*} L^{-1}
$$
and
$$ 
(1 \times m)^* L \hat{*} L \otimes p_{12}^*L \hat{*} L^{-1} \otimes p_{13}^* L \hat{*} L^{-1};
$$
these sections are compatible with the canonical commutative diagram
$$ s_1 = s_2.$$
Further, since these sections correspond to the partial group law structure on $L \hat{*} L$, considered as a biextension, they are commutative and associative.  Put
$$ t = x_1^{-1}(s_1) = x_2^{-1}(s_2).$$
The point now is to check that this section of $\Theta(L)$ determines a cubic structure.

For example, over $\K^4$, we consider the $\G_m$-torsor
$$ p_{234}^* \Theta(L) \otimes m_{12}^* \Theta(L)^{-1} \otimes m_{23}^* \Theta(L) \otimes p_{124} \Theta(L)^{-1},$$
which has fiber
$$ 
\Theta(L)|_{y,z,w} \otimes \Theta(L)^{-1}|_{x+y, z, w} \otimes \Theta(L)|_{x,y+z,w} \otimes \Theta(L)^{-1}|_{x,y,w} 
$$
over 
$$(x,y,z,w) \in \K^4.$$  
(Here, we also use subscripts to denote the evident maps determined by the various factors of $\K^4$.)  This above torsor has a canonical section which also induces a canonical isomorphism
$$ \Theta(L)|_{x+y,z,w} \otimes \Theta(L)|_{x,y,w} \xrightarrow{\sim} \Theta(L)|_{y,z,w} \otimes \Theta(L)|_{x,y+z,w}.$$
As in \cite[Section 2.5]{Breen:1983}, the condition that $s_1$ is associative, for example, is equivalent to the assertion that the section
$$ t(y,z,w) \cdot t(x+y,z,w)^{-1} \cdot t(x,y+z,w) \cdot t(x,y,w)^{-1}$$
(over $\K^4$) coincides with the canonical section, that is
$$ t(x+y,z,w) \cdot t(x,y,w) = t(y,z,w) \cdot t(x, y+z,w).$$
Similarly, the associativity of $s_2$ is expressed in terms of the canonical isomorphism
$$ \Theta(L)|_{x,y+z,w} \otimes \Theta(L)|_{x,y,z} \xrightarrow{\sim} \Theta(L)|_{x,z,w} \otimes \Theta(L)|_{x,y,z+w}$$
which pertains to
$$ t(x,y+z,w) \cdot t(x,y,z) = t(x,z,w) \cdot t(x,y,z+w).$$
Finally, as in \cite[page 18]{Breen:1983}, we have an additional \emph{canonical commutative diagram}
$$
\xymatrix{
 \Theta(L)|_{y,z,w} \otimes \Theta(L)|_{x,y+z,w} \otimes \Theta(L)|_{x,y,z} \ar[dr]^-{\sim} \\
\Theta(L)|_{x+y,z,w} \otimes \Theta(L)|_{x,y,w} \otimes \Theta(L)|_{x,y,z} \ar[r]^-{\sim} \ar[u]^-{\wr} &\Theta(L)|_{y,z,w} \otimes \Theta(L)|_{x,z,w} \otimes \Theta(L)|_{x,y,z+w}.
 }$$
This diagram implies that
 $$ t(x+y,z,w) \cdot t(x,y,w) \cdot t(x,y,z) = t(y,z,w) \cdot t(x,z,w) \cdot t(x,y,z+w).$$
 (Here  we also use compatibility of the partial group laws $s_1$ and $s_2$.)

It remains to check invariance under the natural action of $\mathfrak{S}_3$, the symmetric group on three letters.  To this end, for each $\sigma \in \mathfrak{S}_3$, we have the canonical isomorphism
$$ x_\sigma \colon \sigma^* \Theta(L) \xrightarrow{\sim} \Theta(L)$$
which is induced by the group law on $\K$.  Put
$$t = t(x_1,x_2,x_3)$$ 
and let
$$t(x_{\sigma(1)}, x_{\sigma(2)}, x_{\sigma(3)}) = \sigma^* t$$ 
be the section of $\sigma^* \Theta(L)$ which is induced by $t$.  The fact that $s_1$ and $s_2$ are commutative, then imply the relations
$$ 
x_{(12)} (t (x_2, x_1, x_3)) = t(x_1, x_2, x_3)
$$
and
$$
x_{(23)} (t(x_1, x_3, x_2)) = t(x_1, x_2, x_3).
$$
More generally, it holds true that
$$ 
x_\sigma(t(x_{\sigma(1)}, x_{\sigma(2)}, x_{\sigma(3)})) = t(x_1, x_2, x_3),
$$
for each permutation $\sigma \in \mathfrak{S}_3$.  In particular, the pair $(L,t)$ is an object of $\mathcal{C}ub(\K, \G_m)$.
\end{proof}

Next, we consider central extensions.  
By a \emph{theta group} over $\K$, we mean a central extension
\begin{equation}\label{theta:eqn1}
1 \rightarrow \G_m \rightarrow \G \rightarrow \K \rightarrow 0
\end{equation}
of group schemes.  Each such extension determines a well-defined alternating form
\begin{equation}\label{theta:eqn2}
e \colon \K \times \K \rightarrow \G_m
\end{equation}
which is determined by
$$
e(a',b') = [a,b] = ab a^{-1} b^{-1}\text{,} 
$$
for $a,b \in \G$ lying over $a', b' \in \K$.

Here, $[\cdot,\cdot]$ is the commutator of the extension \eqref{theta:eqn1}.  In case that the form \eqref{theta:eqn2} is a perfect pairing, then we say that the extension \eqref{theta:eqn1} is \emph{nondegenerate}.  In this case, the group scheme $\K$ is finite and the pairing \eqref{theta:eqn2} induces an isomorphism amongst $\K$ and its Cartier dual.
The collection of such central extensions \eqref{theta:eqn1}, for $\K$ a given commutative group scheme, form a category which we denote by $\mathcal{C}ent(\K,\G_m)$.  The categories $\mathcal{C}ent(\K,\G_m)$ and $\mathcal{C}ub(\K,\G_m)$ are related as follows.

\begin{proposition}[{\cite[page 24]{Moret-Bailly}}]\label{theta:prop}  Let $\K$ be a commutative group scheme.
The category $\mathcal{C}ent(\K,\G_m)$ is equivalent to the category of couples $((L,t),\sigma)$ where $(L,t)$ is an object of $\mathcal{C}ub(\K,\G_m)$ and where $\sigma$ is a trivialization of the corresponding symmetric biextension structure for $L \hat{*} L$, as a $\G_m$-torsor over $\K \times \K$.
\end{proposition}

\begin{proof}  This is a special case of  \cite[page 22-24]{Moret-Bailly}, compare also with \cite[Section 2]{Breen:1983}.  Here we indicate the main points of the argument.  With this in mind, suppose that $(L,t)$ is a cubic $\G_m$-torsor over $\K$.  Let $\sigma$ be a trivialization of $\Theta_2(L)$.  For each $x \in \K$ and each $y \in \K$, we have a section $\sigma(x,y)$ of 
$$ 
L|_{x+y} \otimes L^{-1}|_x \otimes L^{-1}|_y
$$
together with an isomorphism
\begin{equation}\label{theta:prop:eqn1}
L|_x \otimes L|_y \xrightarrow{\sim} L|_{x+y}.
\end{equation}
This isomorphism \eqref{theta:prop:eqn1} then defines a product law
\begin{equation}\label{theta:prop:eqn2}
* : L \times L \rightarrow L
\end{equation}
that is compatible with addition and the given $\G_m$-action on $L$.  The product law \eqref{theta:prop:eqn2} can be described explicitly in terms of the trivialization $\sigma(\cdot,\cdot)$.  

For instance, if $x,y \in \K$, $u \in L|_x$ and $v \in L|_y$, then
$$
u * v \in L|_{x+y}
$$
is defined by the condition that
$$
\sigma(x,y) = (u * v) \otimes u^{-1} \otimes v^{-1} \in \Theta_2(L)|_{x+y}.
$$
By assumption, the section $\sigma$ is compatible with the biextension structure $t$ of $\Theta_2(L)$.

From this viewpoint, that $\sigma$ is compatible with the first composition law means that
$$
\sigma(x,z) \otimes \sigma(y,z) \in \Theta_2(L)|_{x,z} \otimes \Theta_2(L)|_{y,z}
$$
is sent to
$$
\sigma(x+y,z) \in \Theta_2(L)|_{x+y,z}
$$
via the condition that
$$
(u * w) \otimes u^{-1} \otimes w^{-1} \otimes (v * w ) \otimes v^{-1} \otimes w^{-1} \otimes t(x,y,z) = ((u * v) * w) \otimes (u*v)^{-1} \otimes w^{-1}.
$$
Here, we have that $u \in L|_x$, $v \in L|_y$, $w \in L|_z$ with $x \in \K$, $y \in \K$ and $z \in \K$.

We then note
\begin{equation}\label{theta:prop:eqn8}
t(x,y,z) = ((u*v)*w)\otimes(u*v)^{-1} \otimes (u*w)^{-1} \otimes (v * w)^{-1}  \otimes u \otimes v \otimes w.
\end{equation}
Similarly, compatibility with the second group law means
\begin{equation}\label{theta:prop:eqn9}
t(x,y,z) = (u * (v * w)) \otimes (u * v)^{-1} \otimes (u * w)^{-1} \otimes (v * w)^{-1} \otimes u \otimes v \otimes w.
\end{equation}

Finally, by comparison
\begin{equation}\label{theta:prop:eqn10}
(u * v) * w = u * (v * w).
\end{equation}
In particular, \eqref{theta:prop:eqn10} implies that the product law $*$ on $L \times L$ so defined is associative.  Further, the section of $L$ at the origin induced by the cubic structure $t$ is trivial; the composition
$$
\G_m \otimes L|_y \xrightarrow{\sim} L|_0 \otimes L|_y \xrightarrow{*} L|_y
$$
is the natural action of $\G_m$ on $L|_y$.

In brief, we have shown that the cubic torsor $(L, t)$ over $\K$ is given, via the trivialization $\sigma$, the structure of a group.  More precisely, $L$ is an object of $\mathcal{C}ent(\K,\G_m)$.  Finally, the product law $*$ can be viewed as a left action of $L$ on itself.  The formulas \eqref{theta:prop:eqn8} and \eqref{theta:prop:eqn9} allow for the cubic structure $t$ to be reconstructed from this group law together with the fixed trivialization $\sigma$.  The equivalence desired by the proposition now follows. 
\end{proof}

Proposition \ref{theta:prop} allows for an alternative formulation of Theorem \ref{homog:characterize} and Corollary \ref{anti:affine:cor}.

\begin{theorem}\label{cubic:homog:bundles}
Let $A$ be an Abelian variety over an algebraically closed field $\kk$.  Fix an integer $n$ that is not divisible by the characteristic of $\kk$.  Then the homogeneous rank $n$ Severi-Brauer varieties over $A$ are classified by the choice of an anti-affine extension 
\begin{equation}\label{anti:affine:ext:1}
1 \rightarrow H \rightarrow G \rightarrow A \rightarrow 1
\end{equation}
together with a pair $((L,t),\sigma)$ where $(L,t)$ is an object of $\mathcal{C}ub(H,\G_m)$ and where $\sigma$ is a trivialization of the corresponding symmetric biextension structure for $L \hat{*} L$ as a $\G_m$-torsor over $H \times H$.
\end{theorem}

\begin{proof}
By Theorem \ref{homog:characterize}, the homogeneous $\PP^{n-1}$-bundles over $A$ are classified by the choice of such an anti-affine extension \eqref{anti:affine:ext:1} together with a faithful projective representation 
\begin{equation}\label{project:rep:1}
\rho \colon H \rightarrow \operatorname{PGL}_n.
\end{equation}
But such a faithful projective representation \eqref{project:rep:1} is equivalent to the choice of a central extension of the form \eqref{theta:eqn1} with $\K$ replaced by $H$.  Thus the conclusion desired by   Theorem \ref{cubic:homog:bundles} follows from that of Proposition \ref{theta:prop}.
\end{proof}

Motivated by Proposition \ref{theta:prop} and Theorem \ref{cubic:homog:bundles}, we formulate the concept of \emph{nondegenerate cubic torsor} (see Definition \ref{non-degenerate-couples} below).  The main point is to identify those structures on cubic torsors which correspond to nondegenerate theta groups.  We use this formalism in the statement of Theorem \ref{Br:thm:2}.

\begin{defn}\label{non-degenerate-couples}  Let $\K$ be a finite commutative group scheme with rank $\#\K = n^2$ not divisible by the characteristic of $\kk$.
We say that a couple $((L,t),\sigma)$, where $(L,t)$ is an object of $\mathcal{C}ub(\K,\G_m)$ and where $\sigma$ is a trivialization of the corresponding symmetric biextension $L \hat{*} L$, as a $\G_m$-torsor over $\K \times \K$, is \emph{nondegenerate} if the central extension to which it corresponds, via Proposition \ref{theta:prop}, is nondegenerate.   By abuse of terminology we also simply say that $((L,t),\sigma)$ is a \emph{nondegenerate cubic torsor}.
\end{defn}

\section{Characterizing the homogeneous irreducible Severi-Brauer varieties in terms of cubic structures}\label{S-B-cubical-structures}

Our main goal here is to prove 

\begin{theorem}\label{irred:homog:theta:bundles:thm}  Let $A$ be an Abelian variety over an algebraically closed field $\kk$.  Fix a positive integer $n$, with $p:= \operatorname{char}(\kk) \nmid n$.
Then the homogeneous irreducible $\PP^{n-1}$-bundles over $A$ are classified by pairs 
$$(\K, ((L,t),\sigma))\text{,}$$ 
where $\K \subseteq A[n]$ is a subgroup of rank $n^2$ and where $((L,t), \sigma)$ is a nondegenerate cubic torsor.  In particular, $(L,t)$ is an object of $\mathcal{C}ub(\K,\G_m)$, $\sigma$ is a trivialization of the biextension structure on $\Lambda(L)$ which is induced by $t$ and the central extension corresponding to $((L,t),\sigma)$ is nondegenerate.
\end{theorem}

In our proof of Theorem \ref{irred:homog:theta:bundles:thm}, we require the elementary divisor (or invariant factor) theorem for nondegenerate alternating forms 
$$
e \colon \K \times \K \rightarrow \G_m.
$$ 
Here, $\K \subseteq A[n]$ is a subgroup scheme of rank $n^2$ and $p \nmid n$. 
We recall the main points following \cite[page 294]{MumI} closely (see also \cite[Remark  3.2]{Bri}). 
\begin{itemize}
\item{
Each nondegenerate alternating form 
\begin{equation}\label{theta:non:deg:alt:form}
e \colon \K \times \K \rightarrow \G_m
\end{equation}
determines a sequence of positive integers
$$ \delta = (d_1,d_2,\dots,d_k)\text{,} 
$$ 
with
$$
\text{ $d_{i+1} \mid d_i$, and $d_i > 1$,}
$$
so that 
$$
\K = \bigoplus\limits_{i=1}^k \left(\ZZ / d_i \ZZ \right)^{\oplus 2} \text{.}
$$
}
\item{We can write
$$
\K = \K(\delta) \oplus \K(\delta)^{t}
$$
with
$$ 
\K(\delta) = \bigoplus\limits_{i=1}^k \ZZ / d_i \ZZ
$$
and
$$
\K(\delta)^{t} = \Hom(\K(\delta),\G_m),
$$
the Cartier dual of $\K(\delta)$.  The identification
$$ 
\K(\delta) = \K(\delta)^{t}
$$
is induced by the alternating form $e$.
}
\item{
Put 
$$\G(\delta) = \G_m \times \K(\delta) \times \K(\delta)^{t}$$ 
with group law
$$ (\alpha,x,\ell) \cdot (\alpha',x',\ell') = (\alpha \cdot \alpha'\cdot \ell'(x),x+x',\ell+\ell').
$$
Then $\G(\delta)$ is a nondegenerate central extension
\begin{equation}\label{normal:form:central:extension}
1 \rightarrow \G_m \rightarrow \G(\delta) \rightarrow \K \rightarrow 0;
\end{equation}
it is the nondegenerate central extension determined by the nondegenerate alternating form \eqref{theta:non:deg:alt:form}.  Conversely, the commutator of each nondegenerate central extension
\begin{equation}\label{non:deg:central:ext:theta:structure}
1 \rightarrow \G_m \rightarrow \G \rightarrow \K \rightarrow 0
\end{equation}
determines a nondegenerate alternating form, with shape  \eqref{theta:non:deg:alt:form}.
}
\item{
By considering the vector $\delta$, it follows that the nondegenerate central extension \eqref{non:deg:central:ext:theta:structure}
is equivalent to the nondegenerate central extension \eqref{normal:form:central:extension}.  Finally, our discussion implies, 
that each nondegenerate alternating form \eqref{theta:non:deg:alt:form}
determines an irreducible weight $1$ representation (namely that of $\G(\delta)$ for $\delta$ the elementary divisor vector of $e$).   Conversely, for a given such  vector $\delta$, there is a unique alternating form and irreducible weight $1$ representation of this shape.   
}
\end{itemize}

We now prove Theorem \ref{irred:homog:theta:bundles:thm}.
\begin{proof}[Proof of Theorem \ref{irred:homog:theta:bundles:thm}]
By Proposition \ref{Schrodinger:Azumaya}, the homogeneous irreducible $\PP^{n-1}$-bundles over $A$ have the form 
$$\P = (A^t/\K^t)^t \times^{\K} \PP^{n-1},$$ 
for $\K \subseteq A[n]$ a subgroup with $\# \K = n^2$ and where $\K$ acts on $\PP^{n-1}$ by an irreducible faithful weight one representation.  In particular, such a $\PP^{n-1}$-bundle $\P$ is determined by the weight $1$ representation of a  nondegenerate theta group
$$
1 \rightarrow \G_m \rightarrow \G \rightarrow \K \rightarrow 1.
$$
Thus, in light of Definition \ref{non-degenerate-couples}, Theorem \ref{irred:homog:theta:bundles:thm} then follows upon applying Proposition \ref{cubic:symmetric:biextension:equivalence} and Proposition \ref{theta:prop}.
\end{proof}

Finally, we use Theorem \ref{Brion:Br:thm} and Theorem \ref{irred:homog:theta:bundles:thm} to prove Theorem \ref{Br:thm:2}.

\begin{proof}[Proof of Theorem \ref{Br:thm:2}]
By Theorem \ref{Brion:Br:thm}, each class in $\Br(A)[r]$ is represented by a homogeneous irreducible Severi-Brauer variety.  In turn, by 
Theorem \ref{irred:homog:theta:bundles:thm}, this means that each class in $\Br(A)[r]$ is represented by a nondegenerate pair $((L,t), \sigma)$.   Equivalently, using the language of Definition  \ref{non-degenerate-couples}, each class in $\Br(A)[r]$ is represented by a nondegenerate cubic torsor.
\end{proof}

\providecommand{\bysame}{\leavevmode\hbox to3em{\hrulefill}\thinspace}
\providecommand{\MR}{\relax\ifhmode\unskip\space\fi MR }
\providecommand{\MRhref}[2]{%
  \href{http://www.ams.org/mathscinet-getitem?mr=#1}{#2}
}
\providecommand{\href}[2]{#2}

\end{document}